\documentclass{amsart}
\usepackage{amscd,graphicx,epsfig}
\usepackage[colorlinks=true, pdfstartview=FitV, linkcolor=blue, citecolor=blue, urlcolor=blue]{hyperref}
\usepackage{amssymb}
\usepackage[small,nohug,heads=vee]{diagrams}

\title[A Characterization of Affine $n$-Space]{Automorphism Groups of Affine Varieties and a Characterization of Affine $n$-Space}
\author{Hanspeter Kraft}
\date{January 31, 2015}

\address{Mathematisches Institut,
Universit\"at Basel, Rheinsprung 21, CH-4051 Basel, Switzerland}
\email{Hanspeter.Kraft@unibas.ch}

\thanks{The author was partially supported by Swiss National Science Foundation.}

\newtheorem{thm}{Theorem}[section]
\newtheorem*{thm*}{Theorem}

\newtheorem*{conj*}{Conjecture}

\newtheorem{prop}[thm]{Proposition}
\newtheorem*{prop*}{Proposition}
\newtheorem{lem}[thm]{Lemma}
\newtheorem{cor}[thm]{Corollary}
\newtheorem*{cor*}{Corollary}

\theoremstyle{definition}
\newtheorem{defn}[thm]{Definition}
\newtheorem{exa}[thm]{Example}

\theoremstyle{remark}
\newtheorem*{rem*}{Remark}
\newtheorem{rem}[thm]{Remark}

\newcommand{\name}[1]{\textsc{#1\/}}
\newcommand{\NN}{{\mathbb N}}
\newcommand{\ZZ}{{\mathbb Z}}
\newcommand{\CC}{{\Bbbk}}

\newcommand{\Cst}{{{\CC}^*}}
\newcommand{\Cplus}{{\CC^{+}}}

\newcommand{\Cn}{{\CC^{n}}}

\newcommand{\An}{{\mathbb A}^{n}}

\newcommand{\Aone}{{\mathbb A}^{1}}
\newcommand{\simto}{\xrightarrow{\sim}}
\newcommand{\be}{\begin{enumerate}}
\newcommand{\ee}{\end{enumerate}}

\newcommand{\dxi}{\frac{\partial}{\partial x_{i}}}

\newcommand{\dfx}[2]{\frac{\partial #1}{\partial #2}}

\newcommand{\A}{\mathfrak A}

\newcommand{\quot}{/\!\!/}

\newcommand{\V}{\mathfrak V}

\newcommand{\AAn}{\A_{n}}

\renewcommand{\aa}{\mathfrak a}

\newcommand{\into}{\hookrightarrow}

\newcommand{\Cxn}{\CC[x_{1},\ldots,x_{n}]}

\newcommand{\dd}[1]{\partial_{x_{#1}}}

\newcommand{\tS}{{\tilde S}}

\newcommand{\AAA}{\mathcal A}

\newcommand{\GGG}{\mathcal{G}}
\newcommand{\HHH}{\mathcal H}

\newcommand{\OOO}{\mathcal O}

\newcommand{\SSS}{\mathcal S}

\newcommand{\VVV}{\mathcal V}
\newcommand{\WWW}{\mathcal W}

\newcommand{\Sn}{\SSS_{n}}

\DeclareMathOperator{\End}{End}
\DeclareMathOperator{\id}{id}

\DeclareMathOperator{\Aut}{Aut}
\DeclareMathOperator{\SAut}{SAut}
\DeclareMathOperator{\Lie}{Lie}
\DeclareMathOperator{\VF}{Vec}
\DeclareMathOperator{\Div}{div}

\DeclareMathOperator{\GL}{GL}

\DeclareMathOperator{\SL}{SL}
\DeclareMathOperator{\SLtwo}{SL_{2}(\CC)}

\DeclareMathOperator{\Ad}{Ad}

\DeclareMathOperator{\Int}{Int}
\DeclareMathOperator{\M}{M}
\DeclareMathOperator{\Der}{Der}
\DeclareMathOperator{\Aff}{Aff}

\DeclareMathOperator{\SAff}{SAff}
\DeclareMathOperator{\Jac}{Jac}
\DeclareMathOperator{\pr}{pr}

\DeclareMathOperator{\jac}{jac}

\DeclareMathOperator{\Hom}{Hom}

\newcommand{\VFc}{\VF^{\text{c}}}

\newcommand{\AutL}{\Aut_{\text{\rm Lie}}}
\newcommand{\EndL}{\End_{\text{\rm Lie}}}

\renewcommand{\subset}{\subseteq}
\renewcommand{\supset}{\supseteq}

\newcommand{\g}{\mathbf g}
\newcommand{\f}{\mathbf f}

\renewcommand{\t}{\mathbf t}

\renewcommand{\phi}{\varphi}
\newcommand{\Autalg}{\Aut^{\text{\it alg}}}
\newcommand{\SAutalg}{\SAut^{\text{\it alg}}}
\DeclareMathOperator{\UU}{\mathcal U}

\newcommand{\lab}[1]{\label{#1}\marginpar{{\tiny\color{blue}#1}}}
\newcommand{\rref}[1]{\ref{#1}\marginpar{{\tiny\color{red}#1}}}
\newcommand{\reff}[1]{\ref{#1}\marginpar{{\tiny\color{red}#1}}}
\newcommand{\margin}[1]{\marginpar{{\tiny\color{red}#1}}}

\renewcommand{\lab}[1]{\label{#1}}
\renewcommand{\rref}[1]{\ref{#1}}
\renewcommand{\reff}[1]{\ref{#1}}
\renewcommand{\margin}[1]{}

\begin{document}
\begin{abstract} We show that the automorphism group of affine $n$-space $\An$ determines $\An$ up to isomorphism: If $X$ is a connected affine variety such that $\Aut(X) \simeq \Aut(\An)$ as ind-groups, then $X \simeq \An$ as varieties.

We also show that every finite group and every torus appears as $\Aut(X)$ for a suitable affine variety $X$, but that $\Aut(X)$ cannot be isomorphic to a semisimple group. In fact, if $\Aut(X)$ is finite dimensional and if $X \not\simeq \Aone$, then the connected component $\Aut(X)^{\circ}$ is a torus.

Concerning the structure of $\Aut(\An)$ we
prove that any homomorphism $\Aut(\An) \to \GGG$ of ind-groups either factors through $\jac\colon\Aut(\An) \to \Cst$, or it is a closed immersion. For $\SAut(\An):=\ker(\jac)\subset \Aut(\An)$ we show that every nontrivial homomorphism $\SAut(\An) \to \GGG$ is a closed immersion.

Finally, we prove that every non-trivial homomorphism $\phi\colon\SAut(\An) \to\SAut(\An)$ is an automorphism, and that $\phi$ is given by conjugation with an element from $\Aut(\An)$.
\end{abstract}

\maketitle

%%%%%%%%%%%%%%%%%%%%%%%%%%%%%%%%%%%%%%%%%
\section{Introduction and main results}
Our base field $\CC$ is algebraically closed of characteristic zero. For an affine variety $X$ the automorphism group $\Aut(X)$ has the structure of an {\it ind-group}. We will shortly recall the basic definitions in the following section~\reff{notation.sec}. The classical example is $\Aut(\An)$, the group of automorphism of affine $n$-space $\An = \Cn$. 

The first main result  shows that $\An$ is determined by its automorphism group.

\begin{thm}\lab{thm1}
Let $X$ be a connected affine variety. If  $\Aut(X) \simeq \Aut(\An)$ as ind-groups, then $X \simeq \An$ as varieties.
\end{thm}
It is clear that $X$ has to be connected, since the automorphism group does not change if we form the disjoint union of $\An$ with a variety $Y$ with trivial automorphism group.

Another important question is which groups appear as automorphism groups of affine varieties. For finite groups we have the following result.

\begin{thm}\lab{thm4}
For every finite group $G$ there is a smooth affine curve $C$ such that $\Aut(C)\simeq G$.
\end{thm}

Moreover, there exist surfaces with infinite discrete automorphism groups (see \cite[Proposition~7.5.2]{FuKr2014On-the-geometry-of}). As for algebraic groups, we will see that every torus appears as $\Aut(X)$ (Example~\reff{torus.exa}), but there are no examples where $\Aut(X) \simeq \SLtwo$. In fact, this is not possible as the next result shows.

\begin{thm}\lab{thm2}
Let $X$ be a connected affine variety. If $\dim\Aut(X)<\infty$, then either $X\simeq \Aone$ or $\Aut(X)^{0}$ is a torus.
\end{thm}

The last results concern the automorphism group $\Aut(\An)$ of affine $n$-space. This group has a closed normal subgroup $\SAut(\An)$ consisting  of those automorphism $\f=(f_{1},\ldots,f_{n})$ whose Jacobian determinant $\jac(\f):=\det\left(\frac{\partial f_{i}}{\partial x_{j}}\right)_{(i,j)}$ is equal to 1:
$$
\SAut(\An) := \ker (\jac\colon \Aut(\An) \to \Cst).
$$ 
One could expect that $\SAut(\An)$ is simple as an ind-group, because its Lie algebra is simple, and that $\SAut(\An)$ is the only closed proper normal subgroup of $\Aut(\An)$. This is claimed in \cite{Sh1966On-some-infinite-d,Sh1981On-some-infinite-d}, but the proofs turned out to be not correct (see \cite[Section~10]{FuKr2014On-the-geometry-of}). 
What we can prove here are the following results.

\begin{thm}\lab{thm3}\strut
\be
\item
Let $\phi\colon \Aut(\An) \to \GGG$ be a homomorphism of ind-groups. Then either  $\phi$ factors through $\jac\colon \Aut(\An) \to \Cst$, or $\phi$ is a closed immersion. This means that $\phi(\Aut(\An))\subset \GGG$ is a closed ind-subgroup and  $\phi\colon\Aut(\An) \simto \phi(\Aut(\An))$ is an isomorphism.
\item
Every nontrivial homomorphism $\SAut(\An) \to \GGG$ of ind-groups is a closed immersion. 
\ee
\end{thm}

\begin{thm} \lab{thm5}
\hspace{1em}
\be
\item 
Every injective homomorphism $\phi\colon \Aut(\An) \to \Aut(\An)$ is an isomorphism, and $\phi = \Int\g$ for a well-defined $\g \in \Aut(\An)$.
\item
Every nontrivial homomorphism $\phi\colon \SAut(\An) \to \SAut(\An)$ is an isomorphism, and $\phi = \Int\g$ for a well-defined $\g \in \Aut(\An)$.
\ee
\end{thm}

Let us point out the following example from \cite{FuKr2014On-the-geometry-of} showing that bijective homomorphisms of ind-groups are not necessarily isomorphisms. Denote by $\CC\langle x,y \rangle$ the free associative $\CC$-algebra in two generators. Then $\Aut(\CC\langle x,y \rangle)$ is an ind-group, and we have a canonical homomorphism $\pi\colon \Aut(\CC\langle x,y \rangle) \to \Aut(\CC[x,y])$. 

\begin{prop}[\name{Furter-Kraft} \cite{FuKr2014On-the-geometry-of}]\label{FK.prop}
The map $\pi\colon \Aut(\CC\langle x,y \rangle) \to \Aut(\CC[x,y])$ is a bijective homomorphism of ind-groups, but it is not an isomorphism, because it is not an isomorphism on the Lie algebras.
\end{prop}

\par\medskip
%%%%%%%%%%%%%%%%%%%%%%%%%%
\section{Notation and preliminary results}\lab{notation.sec}

The notion of an ind-group goes back to Shafarevich who called these objects {\it infinite dimensional groups}, see \cite{Sh1966On-some-infinite-d,Sh1981On-some-infinite-d}). We refer to \cite{FuKr2014On-the-geometry-of} and  \cite{Ku2002Kac-Moody-groups-t} for basic notation in this context and to \cite{KrZa2014Locally-finite-gro} and \cite{KrRe2014Automorphisms-of-t} for some important results.

\begin{defn}\lab{indvar.def}
An {\it ind-variety} $\VVV$ is a set together with an ascending filtration $\VVV_{0}\subset \VVV_{1}\subset \VVV_{2}\subset \cdots\subset \VVV$ such that the following holds:
\be
\item $\VVV = \bigcup_{k \in \NN}\VVV_{k}$;
\item Each $\VVV_{k}$ has the structure of an algebraic variety;
\item For all $k \in \NN$ the subset  $\VVV_{k}\subset \VVV_{k+1}$ is closed in the Zariski-topology.
\ee
\end{defn}

A {\it morphism} between ind-varieties $\VVV=\bigcup_{k}\VVV_{k}$ and $\WWW=\bigcup_{m}\WWW_{m}$  is a map $\phi\colon \VVV \to \WWW$  such that for every $k$ there is an $m$ such that $\phi(\VVV_{k}) \subset \WWW_{m}$ and that the induced map $\VVV_{k}\to \WWW_{m}$ is a morphism of varieties. {\it Isomorphisms} of ind-varieties are defined in the usual way.

Two filtrations $\VVV = \bigcup_{k \in \NN} \VVV_{k}$ and $\VVV = \bigcup_{k \in \NN} \VVV_{k}'$ are called {\it equivalent\/} if for any $k$ there is an $m$ such that $\VVV_{k}\subset \VVV_{m}'$ is a closed subvariety as well as $\VVV_{k}'\subset \VVV_{m}$. Equivalently,  the identity map $\id \colon \VVV = \bigcup_{k \in \NN} \VVV_{k} \to \VVV = \bigcup_{k \in \NN} \VVV_{k}'$ is an isomorphism of ind-varieties.

An ind-variety $\VVV$ has a natural topology where $S\subset \VVV$ is open, resp. closed, if  $S_{k}:=S \cap \VVV_{k} \subset \VVV_{k}$ is open, resp. closed, for all $k$. Obviously, a locally closed subset $S \subset \VVV$ has a natural structure of an ind-variety. It is called an {\it ind-subvariety}. An ind-variety $\VVV$ is called {\it affine} if all $\VVV_{k}$ are affine. A subset $X\subset \VVV$ is called {\it algebraic} if it is locally closed and contained in some $\VVV_{k}$. Such an $X$ has a natural structure of an algebraic variety.

\begin{exa} \lab{countable vector space.exa}
Any $\CC$-vector space $V$ of countable dimension carries the structure of an (affine) ind-variety by choosing an increasing sequence of finite dimensional subspaces $V_{k}$ such that $V = \bigcup_{k} V_{k}$. Clearly, all these filtrations are equivalent. 

If $R$ is a commutative $\CC$-algebra of countable dimension, $\aa \subset R$ a subspace, e.g. an ideal, and $S\subset\CC[x_{1},\ldots,x_{n}]$ a set of polynomials, then the subset 
\[ \{(a_{1},\ldots,a_{n})\in R^{n}\mid f(a_{1},\ldots,a_{n})\in\aa \text{ for all }f\in S\}\subset R^{n}\] 
is a closed ind-subvariety of $R^{n}$.
\end{exa}

For any ind-variety $\VVV = \bigcup_{k\in\NN}\VVV_{k}$ we can define the tangent space in $x\in \VVV$ in the obvious way. We have $x \in \VVV_{k}$ for $k \geq k_{0}$, and $T_{x}\VVV_{k}\subseteq T_{x}\VVV_{k+1}$ for $k\geq k_{0}$, and then define
$$
T_{x}\VVV := \varinjlim_{k\geq k_{0}} T_{x}\VVV_{k}
$$
which is a vector space of countable dimension. A morphism $\phi\colon \VVV \to \WWW$ induces linear maps $d\phi_{x}\colon T_{x}\VVV \to T_{\phi(x)}\WWW$ for every $x\in X$. Clearly, for a $\CC$-vector space $V$ of countable dimension and a for any $v\in V$ we have $T_{v}V = V$ in a canonical way. 

The {\it product} of two ind-varieties is defined in the obvious way. This allows to define an {\it ind-group} as an ind-variety $\GGG$ with a group structure such that multiplication $\GGG\times\GGG \to \GGG\colon (g,h)\mapsto g\cdot h$, and inverse $\GGG \to \GGG\colon g\mapsto g^{-1}$, are both morphisms.
It is clear that a closed subgroup $G$ of an ind-group $\GGG$ is an algebraic group if and only if $G$ is an algebraic subset of $\GGG$.

If $\GGG$ is an affine ind-group, then $T_{e}\GGG$ has a natural structure of a Lie algebra which will be denoted by $\Lie \GGG$. The structure is obtained by showing that every $A\in T_{e}\GGG$ defines a unique left-invariant vector field $\delta_{A}$ on $\GGG$, see \cite[Proposition 4.2.2, p. 114]{Ku2002Kac-Moody-groups-t}. 

\begin{rem}\lab{simple.rem}
It is known that for $n\geq 2$ the Lie algebra $\Lie\SAut(\An)$ is simple and that $\Lie\SAut(\An) \subset \Lie\Aut(\An)$ is the only proper ideal, see \cite[Lemma~3]{Sh1981On-some-infinite-d}. Moreover, both Lie algebras are generated by the subalgebras $\Lie G$ where $G$ is an algebraic subgroup. 
\end{rem}

\begin{defn} 
An ind-group $\GGG=\bigcup_{k}\GGG_{k}$ is called {\it discrete\/} if $\GGG_{k}$ is finite for all $k$.
Clearly, $\GGG$ is discrete if and only if $\Lie\GGG$ is trivial.
\par\smallskip
An ind-group $\GGG$ is called {\it connected} if for every $g\in\GGG$ there is an irreducible curve $D$ and a morphism $D \to \GGG$ whose image contains $e$ and $g$.
\end{defn}

The next result follows from \cite[Theorem~3.1.1]{FuKr2014On-the-geometry-of} and \cite[Theorem~4.3.2]{KrZa2014Locally-finite-gro}. Here $\VF(X)$ denotes the Lie algebra of (algebraic) vector fields on $X$, i.e. $\VF(X) = \Der(\OOO(X))$, the Lie algebra of derivations of $\OOO(X)$.

\begin{prop}\lab{AutX.prop}
Let $X$ be an affine variety. Then $\Aut(X)$ has a natural structure of an affine ind-group, 
and there is a canonical embedding $\xi\colon\Lie\Aut(X) \into \VF(X)$  of Lie algebras.
\end{prop}

\begin{rem}\lab{LieAutAn.rem}
In case $X = \An$ the embedding $\xi$ identifies $\Lie\Aut(\An))$ with $\VFc(\An)$, the vector fields $\delta = \sum_{i}f_{i}\dxi$ with constant divergence $\Div \delta = \sum_{i}\dfx{f_{i}}{x_{i}} \in \CC$, see \cite[Proposition~3.5.1]{FuKr2014On-the-geometry-of}.
\end{rem}

Another result which we will need is proved in \cite[Proposition~6.5.2]{KrZa2014Locally-finite-gro}.
\begin{prop}\lab{dphi.prop}
Let $\phi,\psi\colon \GGG \to \HHH$ be two homomorphisms of ind-groups. Assume that $\GGG$ is connected and that $d\phi_{e} = d\psi_{e}\colon \Lie\GGG \to \Lie\HHH$. Then $\phi = \psi$.
\end{prop}

A final result which we will use is the following (see \cite[Lemma~6.1]{KrZi2014Small-affine-varie}). Denote by {\it $\Aff_{n}\subset \Aut(\An)$ the  subgroup of affine transformations}, i.e. $\Aff_{n}=\GL_{n}(\CC)\ltimes (\Cn)^{+}$. Similarly, the subgroup $\SAff_{n}\subset \Aff_{n}$ consists of the transformations with determinant $1$, i.e. $\SAff_{n}=\SL_{n}(\CC)\ltimes (\Cn)^{+}$.

\begin{prop}\lab{Aff.prop}
Let $X$ be an affine variety with a faithful action of $\SAff_{n}$. If $\dim X \leq n$, then $X$ is $\SAff_{n}$-isomorphic to $\An$.
\end{prop}

\begin{rem}
It is shown in \cite[Lemma~6.1]{KrZi2014Small-affine-varie} that the same holds if we replace $\SAff_{n}$ by $\Aff_{n}$. Using Theorem~\rref{thm5} we see that we can replace $\SAff_{n}$ by $\Aut(\An)$ or $\SAut(\An)$ as well.
\end{rem}

\par\smallskip
%%%%%%%%%%%%%%%%%%%%%%%%%%
\section{The adjoint representation}
If $L$ is a finitely generated Lie algebra, then $\AutL(L)$ has a natural structure of an ind-group defined in the following way (see \cite{FuKr2014On-the-geometry-of}). Choose a finite-dimensional subspace $L_{0} \subset L$ which generates $L$ as a Lie algebra. Then the restriction map $\EndL(L) \to \Hom(L_{0},L)$ is injective and the image is a closed affine ind-subvariety. (To see this write $L$ as the quotient of the free Lie algebra $F(L_{0})$ over $L_{0}$ modulo an ideal $I$.) Choosing a filtration $L = \bigcup_{k\geq 0}L_{k}$ by finite-dimensional subspaces, we set $\EndL(L)_{k}:=\{\alpha\in\EndL(L) \mid \alpha(L_{0})\subset L_{k}\}$ which is a closed subvariety of $\Hom(L_{0},L_{k})$ (see Example~\ref{countable vector space.exa}).
Then we define the ind-structure on $\AutL(L)$ by identifying $\AutL(L)$ with the closed subset 
$$
\{(\alpha,\beta) \in \EndL(L)\times\EndL(L) \mid \alpha\circ \beta = \beta\circ \alpha = \id_{L}\} \subset \EndL(L)\times\EndL(L),
$$
i.e. 
$$
\AutL(L)_{k}:=\{\alpha\in\AutL(L) \mid  \alpha,\alpha^{-1}\in\EndL(L)_{k}\}.
$$
It follows that $\AutL(L)$ is an affine ind-group with the usual functorial properties.
\begin{lem}\lab{indrep.lem}
Let $\rho\colon \GGG \to \AutL(L)$ be an abstract homomorphism where $L$ is a finitely generated Lie algebra. Then $\rho$ is a homomorphism of ind-groups if and only if $\rho$ is an ind-representation, i.e. the map $\rho\colon \GGG \times L \to L$ is a morphism of ind-varieties.
\end{lem}
\begin{proof} 
Assume that $L$ is generated by the finite dimensional subspace $L_{0}\subset L$.
If $\GGG = \bigcup_{j}\GGG_{j}$ and if $\rho\colon \GGG \times L \to L$ is a morphism, then, for any $j$, there is a $k=k(j)$ such that $\rho(\GGG_{j}\times L_{0})\subset L_{k}$ and $\rho(\GGG_{j}^{-1}\times L_{0}) \subset L_{k}$. Hence, $\rho(\GGG_{j})\subset \AutL(L)_{k}$, and the map $\GGG_{j}\to \Hom(L_{0},L_{k})$ is clearly a morphism.

Now assume that $\GGG \to \AutL(L)$ is a homomorphism of ind-groups. Then, for any $j$, there is a $k=k(j)$ such that $\rho(\GGG_{j}) \subset \AutL(L)_{k}\into \Hom(L_{0},L_{k})$. Hence, $\rho(\GGG_{j}\times L_{0})\subset L_{k}$, and $\GGG_{j}\times L_{0} \to L_{k}$ is a morphism.
\end{proof}
The {\it adjoint representation\/} $\Ad\colon \GGG \to \AutL(\Lie\GGG)$ of an ind-group $\GGG$ is defined in the usual way: $\Ad g := (d\Int  g)_{e}\colon\Lie\GGG \simto\Lie\GGG$ where $\Int g$ is the inner automorphism $ h\mapsto  g  h  g^{-1}$.
\begin{prop} 
For any ind-group $\GGG$ the canonical map $\Ad\colon \GGG \to \AutL(\Lie\GGG)$ is a homomorphism of ind-groups.
\end{prop}
\begin{proof} Let $\gamma\colon\GGG\times\GGG \to \GGG$ denote the morphism $(g,h)\mapsto g h g^{-1}$.
For any $g \in \GGG$, the map $\gamma_{g}\colon\GGG \to \GGG$, $h\mapsto g h g^{-1}$, is an isomorphism of ind-groups, and its differential $\Ad(g)=(d\gamma_{g})_{e}\colon\Lie\GGG \to \Lie\GGG$ is an isomorphism of Lie algebras. If $\GGG = \bigcup_{k}\GGG_{k}$, then for any $p,q \in\NN$ there is an $m\in\NN$ such that $\gamma \colon \GGG_{p}\times \GGG_{p} \to \GGG_{m}$. Clearly, $\Ad g$ for $g \in \GGG_{p}$ is then given by $(d\gamma_{g})_{e}\colon T_{e}\GGG_{q}\to T_{e}\GGG_{m}$, and the map $\GGG_{k}\to\Hom(T_{e}\GGG_{q},\GGG_{m})$ is a morphism, by the following lemma. Now the claim follows from Lemma~\ref{indrep.lem}.
\end{proof}

\begin{lem}\lab{Ad.lem}
Let $\Phi\colon X \times Y \to Z$ be a morphism of affine varieties and  set $\Phi_{x}(y):=\Phi(x,y)$. Assume that there exist $y_{0}\in Y$ and $z_{0}\in Z$ such that $\Phi_{x}(y_{0}) = z_{0}$ for all $x\in X$. Then the induced map $X \to \Hom(T_{y_{0}}Y, T_{z_{0}}Z)$, $x \mapsto d_{y_{0}}\Phi_{x}$,  is a morphism.
\end{lem}
\begin{proof}
We can assume that $Y,Z$ are vector spaces, $Y = W$ and $Z=V$. Choose bases  $(w_{1},\ldots,w_{m})$ and $(v_{1},\ldots, v_{n})$. Then $\Phi$ is given by an element  of the form
$$
\sum_{i=1}^{n}\sum_{j} f_{ij} \otimes h_{ij} \otimes v_{i}, \text{ where } f_{ij}\in\OOO(X) \text{ and }  h_{ij}\in\OOO(Y)=\CC[y_{1},\ldots,y_{m}],
$$
and so the differential $(d\Phi_{x})_{y_{0}} \colon W \to V$ is given by the matrix 
$$
\left(\sum_{j}f_{ij}(x) \frac{\partial h_{ij}}{\partial y_{k}}(y_{0})\right)_{(i,k)}
$$
whose entries are regular functions on $x$. The claim follows.
\end{proof}
\begin{rem}\lab{Ad-is-bij.rem}
In \cite{KrRe2014Automorphisms-of-t} we show that the canonical homomorphisms
\[
\Ad\colon\Aut(\An) \to \AutL(\Lie \Aut(\An)) \text{ and }  \Ad\colon\Aut(\An) \to \AutL(\Lie \SAut(\An))
\]
are both bijective. This can be improved.
\end{rem}

\begin{prop}\lab{Ad-is-iso.prop}
\hspace{1em}
\be
\item
The adjoint representation 
$
\Ad\colon \Aut(\An) \to \AutL(\Lie\Aut(\An))
$
is an isomorphism of ind-groups.
\item
The induced map $\rho\colon \AutL(\Lie\Aut(\An)) \to \AutL(\Lie\SAut(\An))$ is an isomorphism of ind-groups.
\ee
\end{prop}
\begin{proof}
We will use here the identification of $\Lie\Aut(\An)$ with $\VFc(\An)$, see Remark~\ref{LieAutAn.rem}. Put $\dd{i}:=\dxi$.

(1)
Let $\f=(f_{1},\ldots,f_{n}) \in \Aut(\An)$ and set $\theta:=\Ad(\f^{-1})\in\AutL(\VFc(\An))$. Then the matrix $\Bigl(\theta(\dd{k})x_{j}\Bigr)_{(j,k)}$ is invertible, and 
\[\tag{$*$}
\Bigl(\theta(\dd{k})x_{j}\Bigr)_{(j,k)}^{-1} = \Jac(\f)=\Bigl(\frac{\partial f_{j}}{\partial x_{i}}\Bigr)_{(i,j)},
\] 
see \cite[Remark~4.2]{KrRe2014Automorphisms-of-t}.
We now claim that  the map 
$$
\theta\mapsto \Bigl(\theta(\dd{k})x_{j}\Bigr)_{(j,k)}^{-1}: \AutL(\VFc(\An)) \to \M_{n}(\Cxn)
$$ 
is a well-defined morphism of ind-varieties. In fact, $\theta\mapsto \theta(\dd{k})x_{j}$ is the composition of the orbit map $\theta\mapsto \theta(\dd{k})
\colon \AutL(\VFc(\An)) \to \VFc(\An)$ and the evaluation map $\delta\mapsto \delta(x_{j})\colon \VFc(\An) \to \Cxn$, hence $\theta\mapsto \Theta:=\Bigl(\theta(\dd{k})x_{j}\Bigr)_{(j,k)}$ is a morphism. Since $\jac(\Theta)\in\Cst$ the claim follows. 

Now recall that the gradient 
$\Cxn \to \Cxn^{n}, \ f\mapsto (\dfx{f}{x_{1}},\ldots,\dfx{f}{x_{n}})$,  defines an isomorphism
$$
\gamma\colon \Cxn_{\geq 1} \simto \Gamma:=\{(h_{1},\ldots,h_{n})\mid \dfx{h_{i}}{x_{j}} = \dfx{h_{j}}{x_{i}} \text{ for all } i<j\}.
$$
It follows from $(*)$ that the rows of the matrix $(h_{ij})_{(i,j)} := \Bigl(\theta(\dd{k})x_{j}\Bigr)_{(j,k)}^{-1}$ belong to $\Gamma$, so that we  get a morphism 
$$
\psi\colon \AutL(\VFc(\An)) \to \Cxn^{n}, \  \theta\mapsto (f_{1},\ldots,f_{n}),
$$
where $f_{i}:=\gamma^{-1}(h_{i1},\ldots,h_{in})\in\Cxn_{\geq 1}$. By construction, we have 
\[\tag{$**$}
\psi(\theta) = \psi(\Ad(\f^{-1})) = \f_{0}:=(f_{1}-f_{1}(0),\ldots,f_{n}-f_{n}(0)) = \t_{-\f(0)} \circ \f
\]
where $\t_{a}$ is the translation $v \mapsto v+a$.
Let $S \subset \Aff_{n}$ be the subgroup of translations, and set $\tS := \Ad(S)$. Then $\tS \subset \AutL(\VFc(\An))$ is a closed algebraic subgroup and $\Ad\colon S \to \tS$ is an isomorphism. 
It follows from $(**)$ that $\Ad(\psi(\theta))\cdot \theta = \Ad(\t_{-\f(0)}) \in \tS$, and so 
$$
\tilde\psi(\theta):= \psi(\theta)^{-1}\cdot (\Ad|_{S})^{-1}(\Ad(\psi(\theta))\cdot \theta)
$$ 
is a well-defined morphism $\tilde\psi \colon \AutL(\VFc(\An)) \to \Aut(\An)$ with the property that 
$$
\Ad(\tilde\psi(\theta)) = \Ad(\psi(\theta)^{-1}) \cdot \Ad(\psi(\theta)) \cdot \theta = \theta.
$$ 
Thus $\Ad\colon \Aut(\An) \to \AutL(\Lie\Aut(\An))$ is an isomorphism, with inverse $\tilde\psi$.
\par\smallskip
(2)
Clearly, the restriction $\rho\colon \AutL(\Lie\Aut(\An)) \to \AutL(\Lie\SAut(\An))$ is a homomorphism of ind-groups, and it is bijective (Remark~\reff{Ad-is-bij.rem}). It follows from (1) that the composition $\rho\circ\Ad\colon \Aut(\An)\to \AutL(\Lie\SAut(\An))$ is a bijective homomorphism of ind-groups. Now we use Theorem~\reff{thm3}(1) to conclude that $\rho\circ\Ad$ is an isomorphism, hence $\rho$ is an isomorphism, too. Note that in the proof of Theorem~\ref{thm3}(1) below we will only use Proposition~\ref{Ad-is-iso.prop}(1).
\end{proof}

\begin{proof}[Proof of Theorem~\ref{thm3}]
(1)
Let $\phi\colon \Aut(\An) \to \GGG$ be an homomorphism of ind-groups such that $d\phi$ is injective. We can assume that  $\GGG = \overline{\phi(\Aut(\An))}$, and we will show that $\phi$ is an isomorphism. The basic idea is to construct a homomorphism $\psi\colon \GGG \to \Aut(\An)$ such that $\psi\circ\phi = \id$. By Proposition~\reff{retraction.prop} below this implies that $\phi$ is a closed immersion, hence an isomorphism.

Denote by $L \subset \Lie\GGG$ the image of $d\phi$.
For any $g \in \Aut(\An)$ we have $d\phi\circ\Ad(g) = \Ad(\phi(g)) \circ d\phi$. In particular, $L$ is stable under $\Ad(\phi(g))$, hence stable under $\Ad(\GGG)$, because $\phi(\Aut(\An))$ is dense in $\GGG$. Thus we get the following commutative diagram of homomorphisms of ind-groups
$$
\begin{CD}
\Aut(\An) @>\phi>> \GGG \\
@V{\Ad_{\Aut(\An)}}V{\simeq}V  @VV{\Ad_{\GGG}}V \\
\AutL(\Lie\Aut(\An))) @>\simeq>> \AutL(L)
\end{CD}
$$
where the first vertical map is an isomorphism by Proposition~\reff{Ad-is-iso.prop}(1). Thus, the composition $\Ad_{\GGG}\circ\phi\colon \Aut(\An) \to \AutL(L)\simeq\Aut(\An)$ is an isomorphism, and so $\phi$ is also an isomorphism, by Proposition~\reff{retraction.prop} below.

If $d\phi$ is not injective, then $\ker d\phi\supset \Lie\SAut(\An)$ (Remark~\reff{simple.rem}) and so $d\phi = f\circ d\jac$ where $f\colon \CC \to \Lie\GGG$ is a Lie algebra homomorphism. If $\Cst \subset \GL_{n}(\CC)$ denotes the center, then $\phi|_{\Cst}\colon\Cst \to \GGG$ factor through $?^{n}\colon \Cst \to\Cst$, because $\SL_{n}(\CC)\subset\ker\phi$, i.e. $\phi(z)=\rho(z^{n})$ for any $z\in\Cst$ and a suitable homomorphism $\rho\colon\Cst \to \GGG$ of ind-groups. By construction, $d\rho_{e} = f\colon \CC \to \Lie\GGG$, and so the two homomorphisms $\phi$ and $\rho\circ\jac$ have the same differential. Thus, by Proposition~\reff{dphi.prop}, we get $\phi=\rho\circ\jac$, and we are done.
\par\smallskip
(2)
Let $\phi\colon\SAut(\An) \to \GGG$ be a homomorphism of ind-groups. 
If $d\phi_{e}$ is not injective, then $d\phi_{e}$ is the trivial map (Remark~\reff{simple.rem}), hence $d\phi_{e} = d\bar\phi_{e}$ where $\bar\phi\colon \g \mapsto e$ is the constant homomorphism. Again by Proposition~\reff{dphi.prop} we get $\phi =\bar\phi$. 

If $d\phi_{e}$ is injective, set $L:=d\phi_{e}(\Lie\SAut) \subset \Lie\GGG$. Again we can assume that $\GGG = \overline{\phi(\SAut(\An))}$. Since $L$ is stable under $\Ad\phi(\g)$ for all $\g \in \Aut(\An)$ it is also stable under $\GGG$, and  we get, as above, the following commutative diagram
$$
\begin{CD}
\Aut(\An) @<\supseteq<< \SAut(\An) @>\phi>> \GGG \\
@V{\Ad_{\Aut(\An)}}V{\simeq}V @V{\Ad_{\SAut(\An)}}V{\subset}V  @VV{\Ad_{\GGG}}V \\
\AutL(\Lie\Aut(\An))) @>\Psi>\text{bij}> \AutL(\Lie\SAut(\An))) @>\Phi>\simeq> \AutL(L)
\end{CD}
$$
where $\Ad_{\Aut(\An)}$ is an isomorphism and $\Psi$ is a bijective homomorphism (see Proposition~\reff{Ad-is-iso.prop} and Remark~\reff{Ad-is-bij.rem}).
The image $\A \subset \AutL(\Lie \SAut(\An))$ of $\SAut(\An)$ is a closed subgroup isomorphic to $\SAut(\An)$, and  $\A\simto\Phi(\A) =\Ad_{\GGG}(\phi(\SAut(\An))$. But  $\phi(\SAut(\An)) \subset \GGG$ is dense, and so $\Ad_{\GGG}(\GGG) = \Phi(\A)$. Thus, the composition $\Ad_{\GGG} \circ\phi\colon \SAut(\An)\to \Phi(\A)$ is an isomorphism, and so $\phi$ is an isomorphism, by Proposition~\reff{retraction.prop} below.
\end{proof}

\begin{prop}\lab{retraction.prop}
Let $\HHH,\GGG$ be two ind-groups, and let $\phi\colon \HHH \to \GGG$, $\psi\colon \GGG \to \HHH$ be two homomorphisms. If $\psi\circ\phi = \id_{\HHH}$, then $\phi$ is a closed immersion, i.e. $\phi(\HHH) \subset \GGG$ is a closed subgroup and $\phi$ induces an isomorphism $\HHH \simto \phi(\HHH)$.
\end{prop}
\begin{proof} By base change we can assume that the base field $\CC$ is uncountable.
Let $\HHH = \bigcup_{i}\HHH_{i}$ and $\GGG=\bigcup_{j}\GGG_{j}$ where we can assume that $\HHH_{i}\subset \GGG_{i}$ for all $i$. Moreover, for every $i$ there is a $k = k(i)$ such that $\psi(\GGG_{i}) \subset \HHH_{k}$. By assumption, the composition $\psi\circ\phi\colon \HHH_{i}\to \GGG_{i} \to \HHH_{k}$ is the closed embedding $\HHH_{i} \into \HHH_{k}$, hence the first map is a closed embedding. Thus $H_{i}:=\phi(\HHH_{i})$ is a closed subset of $\GGG_{i}$ and $H:=\phi(\HHH) = \bigcup_{i} H_{i}$. Now the claim follows from Lemma~\reff{closedsubgroup.lem} below by setting $S : =\ker\psi$.
\end{proof}

Recall that a subset $S \subset \V$ of an ind-variety $\V$ is called {\it ind-constructible\/} if $S = \bigcup_{i}S_{i}$ where $S_{i}\subset S_{i+1}$ are constructible subsets of $\V$.

\begin{lem}\lab{closedsubgroup.lem}
Let $\GGG$ be an ind-group, $H \subset \GGG$ a subgroup and $S \subset \GGG$ an ind-constructible subset. Assume that $\CC$ is uncountable and that
\be
\item $H = \bigcup_{i} H_{i}$ where $H_{i}\subset H_{i+1}\subset \GGG$ are closed algebraic subsets,
\item the multiplication map $S \times H \to \GGG$ is bijective. 
\ee
Then $H$ is a closed subgroup of $\GGG$.
\end{lem}
\begin{proof}
Let $\GGG = \bigcup_{k}\GGG_{k}$. We have to show that for every $k$ there exists an $i=i(k)$ such that $H\cap \GGG_{k}= H_{i} \cap \GGG_{k}$. We can assume that  $e\in S = \bigcup_{i} S_{i}$. Then, by assumption, $\GGG = \bigcup_{j} S_{j} H_{j}$. Since $S_{j}H_{j}\cap \GGG_{k}$ is a constructible subset of $\GGG_{k}$ it follows that there exists a $j=j(k)$ such that $\GGG_{k}\subset S_{j} H_{j}$ (\cite[Lemma~3.6.4]{FuKr2014On-the-geometry-of}).
Setting $\dot S := S \setminus\{e\}$ we get $\dot S H \cap H = \emptyset$.
Thus, $\GGG_{k} = (\dot S_{i}H_{i} \cap \GGG_{k}) \cup (H_{i} \cap \GGG_{k})$ and $H\cap \dot S_{i} H_{i}=\emptyset$, hence
$H \cap \GGG_{k} = H_{i}\cap \GGG_{k}$.
\end{proof}

Finally, we can prove Theorem~\reff{thm5}.

\begin{proof}[Proof of Theorem~\ref{thm5}]
(1) We already know from Theorem~\reff{thm3} that an injective homomorphism $\phi\colon \Aut(\An) \to \Aut(\An)$ is a closed immersion. We claim that $d\phi_{e}\colon \Lie\Aut(\An) \to \Lie\Aut(\An)$ is an isomorphism. To show this, consider the linear action of $\GL_{n}(\CC)$ on $\Lie\Aut(\An)$. We then have 
$$
\Lie\Aut(\An) \subset \VF(\An) \simeq \Cn\otimes\CC[x_{1},\ldots,x_{n}]=\bigoplus_{d}\Cn\otimes\CC[x_{1},\ldots,x_{n}]_{d}
$$ 
and the latter is multiplicity-free as a $\GL_{n}(\CC)$-module as well as an $\SL_{n}(\CC)$-module.  

Now $\phi(\GL_{n}(\CC)) \subset \Aut(\An)$ is a closed subgroup isomorphic to $\GL_{n}(\CC)$. 
Moreover, $d\phi_{e}\colon \Lie\Aut(\An) \to \Lie\Aut(\An)$ is an injective linear map which is equivariant with respect to $\phi\colon \GL_{n}(\CC)\simto \phi(\GL_{n}(\CC))$. Since $\phi(\GL_{n}(\CC))$ is  conjugate to the standard $\GL_{n}(\CC)\subset\Aut(\An)$ and since  the representation of $\GL_{n}(\CC)$ on $\Lie\Aut(\An)$ is multiplicity-free, it follows that $d\phi_{e}$ is an isomorphism. Thus $\GGG:=\phi(\Aut(\An))\subset \Aut(\An)$ is a closed subgroup with the same Lie algebra as $\Aut(\An)$, and we get the following commutative diagram (see proof of Theorem~\reff{thm3}):
$$
\begin{CD}
\Aut(\An) @>\phi>> \GGG  @>\subseteq>> \Aut(\An)\\
@V{\Ad_{\Aut(\An)}}V{\simeq}V  @VV{\Ad_{\GGG}}V @VV{\Ad_{\Aut(\An)}}V\\
\AutL(\Lie\Aut(\An))) @>\simeq>> \AutL(\Lie\GGG) @= \AutL(\Lie\Aut(\An))\\
\end{CD}
$$
As a consequence, all maps are isomorphisms, and so $\GGG = \Aut(\An)$ and $\phi$ is an isomorphism. 

It remains to see that every automorphism $\phi\in\Aut(\An)$ is inner. Since $d\phi_{e}\in\AutL(\Lie\Aut(\An))$ we get $d\phi_{e} = \Ad(\g)$ for some $\g \in \Aut(\An)$, by Remark~\reff{Ad-is-bij.rem}. This means that $d\phi_{e}=(d\Int \g)_{e}$ and so $\phi = \Int\g$, by Proposition~\reff{dphi.prop}.
\par\smallskip
(2) 
The same argument as above shows that every nontrivial homomorphism $\SAut(\An) \to \SAut(\An)$ is an isomorphism where we use the fact that the action of $\SL_{n}(\CC)$ on $\Lie\SAut(\An)$ is multiplicity-free. 

Moreover, the homomorphism $\Ad\colon\Aut(\An) \to \AutL(\Lie\SAut(\An))$ is a bijective homomorphism of ind-groups (Remark~\reff{Ad-is-bij.rem}). Hence, for every $\phi\in\SAut(\An)$ there is a $\g\in\Aut(\An)$ such that $d\phi_{e}=\Ad\g$ which implies that $\phi=\Int\g$.
\end{proof}

\par\medskip
%%%%%%%%%%%%%%%%%%%%%%%%%%%%%%%
\section{A special subgroup of $\Aut(X)$}
Our Theorem~\ref{thm1} will follow from a more general result which we will describe now. For any affine variety $X$ consider the normal subgroup $\UU(X)$ of $\Aut(X)$  generated by the unipotent elements, or, equivalently,  by the closed subgroups isomorphic to $\Cplus$. This is an instance of a so-called {\it algebraically generated subgroup\/} of an ind-group, see \cite{KrZa2014Locally-finite-gro}. The group $\UU(X)$ was introduced and studied in \cite{ArFlKa2013Flexible-varieties}  where the authors called it the {\it group of special automorphisms of $X$}. In particular, they proved a very interesting connection between transitivity properties of $\UU(X)$  and the flexibility of $X$.

Let us define the following notion of an ``algebraic'' homomorphism between these groups.
\begin{defn}
A homomorphism $\phi\colon \UU(X) \to \UU(Y)$ is {\it algebraic}, if for any closed subgroup $U \subset \UU(X)$ isomorphic to $\Cplus$ the image $\phi(U) \subset\UU(Y)$ is closed and $\phi|_{U}\colon U \to \phi(U)$ is a homomorphism of algebraic groups. We say that $\UU(X)$ and $\UU(Y)$ are {\it algebraically isomorphic\/}, $\UU(X) \simeq \UU(Y)$,  if there exists a bijective homomorphism $\phi\colon \UU(X) \to \UU(Y)$ such that $\phi$ and $\phi^{-1}$ are both algebraic.
\end{defn}

\begin{lem}
Let $\phi\colon \UU(X) \to \UU(Y)$ be an algebraic homomorphism. Then, for any algebraic subgroup $G \subset \UU(X)$ generated by unipotent elements the image $\phi(G) \subset \UU(Y)$ is closed and $\phi|_{G}\colon G \to \phi(G)$ is a homomorphism of algebraic groups. 
\end{lem}
\begin{proof}
There exist closed subgroups $U_{1},\ldots,U_{m} \subset G$ isomorphic to $\Cplus$ such that the multiplication map $\mu\colon U_{1}\times U_{2}\times\cdots\times U_{m} \to G$ is surjective. This gives the following commutative diagram
\[
\begin{CD}
U_{1}\times U_{2}\times\cdots\times U_{m}  @>\mu>> G \\
@VV{\tilde\phi:=\phi|_{U_{1}}\times\cdots\times\phi|_{U_{m}}}V   @VV{\phi|_{G}}V \\
\phi(U_{1})\times \phi(U_{2})\times\cdots\times \phi(U_{m})  @>\bar\mu>> \phi(G) \\
\end{CD}
\]
where all maps are surjective. It follows that $\overline{\phi(G)} \subset \Aut(Y)$ is a (closed) algebraic subgroup, and thus $\phi(G) = \overline{\phi(G)}$, because $\phi(G)$ is constructible. It remains to show that $\phi|_{G}$ is a morphism. This follows from the next lemma, because $G$ is normal, and $\mu$ and the composition $\phi|_{G} \circ \mu = \bar\mu\circ\tilde\phi$ are both morphisms.
\end{proof}

\begin{lem} Let $X,Y,Z$ be irreducible affine varieties where $Y$ is normal. Let $\mu\colon X \to Y$ be a surjective morphism and $\phi\colon Y \to Z$ an arbitrary map. If the composition $\phi\circ\mu$ is a morphism, then $\phi$ is a morphism.
\end{lem}
\begin{proof}
We have the following commutative diagram of maps
\[
\begin{CD}
\Gamma_{\phi\circ\mu} @>\subset>> X \times Z\\
@VV{\bar\mu}V   @VV{\mu\times\id}V \\
\Gamma_{\phi} @>\subset>> Y \times Z\\
@VV{p}V @V{\pr_{Y}}VV\\
Y @= Y
\end{CD}
\]
where $\Gamma_{\phi\circ\mu}$ and $\Gamma_{\phi}$ denote the graphs of the corresponding maps. We have to show that $\Gamma_{\phi}\subset Y \times Z$ is closed and that $p$ is an isomorphism. The diagram shows that $\bar\mu$ is surjective, hence $\Gamma_{\phi}$ is constructible, and $p$ is bijective. Thus, the induced morphism $\bar p \colon \overline{\Gamma_{\phi}} \to Y$ is birational and surjective, hence an isomorphism  since $Y$ is normal (see \cite[Lemma~4, page~379]{Ig1973Geometry-of-absolu}). Since $p$ is bijective, we finally get $\Gamma_{\phi}=\overline{\Gamma_{\phi}}$.
\end{proof}

The proof of the following theorem will be given in the next section.
\begin{thm}\lab{thm.gen}
Let $X$ be a connected affine variety. If $\UU(X)$ is algebraically isomorphic to $\UU(\An)$, then $X$ is isomorphic to $\An$.
\end{thm}

A first consequence is our Theorem~\ref{thm1}.

\begin{cor}\lab{cor1=thm1}
Let $X$ be a connected affine variety. If  $\Aut(X) \simeq \Aut(\An)$ as ind-groups, then $X \simeq \An$.
\end{cor}
\begin{proof}
It is clear from the definition that an isomorphism $\Aut(X) \simto \Aut(\An)$ induces an algebraic isomorphism $\UU(X) \simto \UU(\An)$.
\end{proof}

Finally, we define the following closed subgroups of $\Aut(X)$:
$$
\Autalg(X):=\overline{\langle G \mid G \subset \Aut(X) \text{ connected algebraic}\rangle},
$$
$$
\SAutalg(X):=\overline{\langle U \mid U \subset \Aut(X) \text{ unipotent algebraic}\rangle}.
$$
We have $\SAutalg(X) = \overline{\UU(X)} \subset \Autalg(X) \subset \Aut(X)$. 
Now the same argument as above gives the next result.
\begin{cor}\lab{cor2}
Let $X$ be a connected affine variety. If $\SAutalg(X)$ is isomorphic to $\SAutalg(\An)$ as ind-groups, then $X$ is isomorphic to $\An$, and the same holds if we replace $\SAutalg$ by $\Autalg$.
\end{cor}

\par\smallskip
%%%%%%%%%%%%%%%%%%%%%%%%%%
\section{Root subgroups and modifications}

Let $\GGG$ be an ind-group, and let $T \subset \GGG$ be a torus.

\begin{defn}
A closed subgroup $U \subset \GGG$ isomorphic to $\Cplus$ and normalized by $T$  is called a {\it root subgroup} with respect to $T$. The character of $T$ on $\Lie U \simeq \CC$ is called the {\it weight of $U$}.
\end{defn}
Let $X$ be an affine variety and consider a nontrivial action of $\Cplus$ on $X$, given by $\lambda\colon\Cplus \to \Aut(X)$. If $f\in\OOO(X)$ is $\Cplus$-invariant, then we define the {\it modification\/} $f\cdot\lambda$ of $\lambda$ in the following way (see \cite[section~8.3]{FuKr2014On-the-geometry-of}):
$$
(f\cdot\lambda)(s)x := \lambda(f(x)s)x \text{ \ for \ }s\in\CC \text{ and }x\in X.
$$
It is easy to see that this is again a $\Cplus$-action. In fact, if the corresponding locally nilpotent vector field is $\delta_{\lambda}$, the $f\delta_{\lambda}$ is again locally nilpotent, because $f$ is $\lambda$-invariant, and defines the modified $\Cplus$-action $f\cdot\lambda$. This modified action $f\cdot\lambda$ is trivial if and only if $f$ vanishes on every irreducible component $X_{i}$ of $X$ where the action $\lambda$ is nontrivial.  It is clear that the orbits of $f\cdot\lambda$ are contained in the orbits of $\lambda$ and that they are equal on the open subset $X_{f} \subset X$. In particular, if $X$ is irreducible and $f\neq 0$, then $\lambda$ and $f\cdot\lambda$ have the same invariants.

If $U \subset \Aut(X)$ is isomorphic to $\Cplus$ and if $f\in\OOO(X)^{U}$  is a $U$-invariant, then we define in a similar way the {\it modification\/} $f\cdot U$ of $U$. Choose an isomorphism $\lambda\colon \Cplus \simto U$ and set $f\cdot U:=(f\cdot\lambda)(\Cplus)$, the image of the modified action. Note that $\Lie (f\cdot U) = f\Lie U \subset \Lie\Aut(X) \subset \VF(X)$ where we use the fact that $\VF(X)$ is a $\OOO(X)$-module.

If a torus $T$ acts linearly and rationally on a vector space $V$ of countable dimension, then we call $V$ {\it multiplicity free} if the weight spaces $V_{\alpha}$ are all of dimension $\leq 1$. The following lemma is crucial.

\begin{lem}\lab{rootdim.lem}
Let $X$ be an irreducible affine variety, and let $T \subset \Aut(X)$ be a torus. Assume that there exists a root subgroup $U \subset \Aut(X)$ with respect to $T$ such that $\OOO(X)^{U}$ is multiplicity-free. Then $\dim T \leq \dim X \leq \dim T + 1$.
\end{lem}
\begin{proof} The first inequality $\dim T \leq \dim X$ is clear, because $T$ acts faithfully on $X$.  It follows from  \cite{KrDu2014Invariants-and-Sep}) that there exists a $T$-semi-invariant $f\in\OOO(X)^{U}$ such that $\OOO(X)_{f}^{U}$ is finitely generated. Clearly,  $\OOO(X)_{f}^{U}$ is $T$-stable and  multiplicity-free.  The algebra $\OOO(X)_{f}^{U}=\OOO(X_{f})^{U}$ is the coordinate ring of the algebraic quotient $Z:=X_{f}\quot U$ on which $T$ acts. It follows from \cite[II.3.4~ Satz 5]{Kr1984Geometrische-Metho}) that $T$ has a dense orbit in $Z$, and so $\dim Z \leq \dim T$. Since $\dim Z = \dim X_{f}\quot U = \dim X_{f} - 1=\dim X -1$, we get the second inequality.
\end{proof}

\begin{lem}\lab{UAn.lem}
We have $\UU(\An) \subset \SAut(\An)$, and its closure $\overline{\UU(\An)}$ is connected. Moreover,  $\Lie \overline{\UU(\An)} = \Lie \SAut(\An)$, hence it is a simple Lie algebra.
\end{lem} 

\begin{proof}
The first statement is obvious, since every unipotent algebraic group is contained in $\SAut(\An)$. The second claim follows from  $\UU(\An) \subset \overline{\UU(\An)}^{\circ}$ (see Lemma~\ref{connectedcomp.lem} in the next section). For the last statement we remark that $\Lie\SAut(\An)$ is generated by the Lie algebras of the algebraic subgroups (Remark~\ref{simple.rem}).
\end{proof}
The group $\UU(\An)$ contains the normal subgroup generated by all tame elements. But we do not know if $\UU(\An) = \SAut(\An)$ or at least $\overline{\UU(\An)} = \SAut(\An)$, except for $n=2$ where this is well-known.

Denote by $T_{n}\subset \GL_{n}(\CC) \subset \Aut(\An)$ the diagonal torus and set $T_{n}':=T_{n}\cap\SL_{n}(\CC)_{n}$. The next result can be found in \cite[Theorem~1]{Li2011Roots-of-the-affin}.

\begin{lem}\lab{rootsAn.lem}
Root subgroups of $\Aut(\An)$ with respect to $T_{n}'$ exist, and they have different weights.
\end{lem}

\begin{proof}[Proof of Theorem~\ref{thm.gen}]
Note that $\SL_{n}(\CC)$ and $\SAff_{n}(\CC)$ both belong to $\UU(\An)$ as well as all root subgroups $U$. Fix an algebraic isomorphism $\phi\colon\UU(\An) \simto \UU(X)$ and denote by $T'$ the image of $T_{n}'$. 

(a) Assume first that $X$ is irreducible.  By Lemma~\rref{rootsAn.lem}, there exists a root subgroup $U\subset \UU(X)$ with respect to $T'$, and all root subgroups have different weights. In particular, the root subgroups from $\OOO(X)^{U}\cdot U \subset \UU(X)$ have different weights which implies that $\OOO(X)^{U}$ is multiplicity-free, because the map $\OOO(X)^{U} \to \OOO(X)^{U}\cdot U$ is injective. Hence, by Lemma~\rref{rootdim.lem}, $\dim X \leq \dim T' + 1=n$, and the claim follows from Proposition~\rref{Aff.prop}. 

\par\smallskip
(b) Let $X = \bigcup_{i}X_{i}$ be the decomposition into irreducible components. Since $\overline{\UU(X)}$ is connected by Lemma~\ref{UAn.lem} it follows that the components $X_{i}$ are stable under $\overline{\UU(X)}$. Moreover, at least one of the restriction morphisms $\rho_{i}\colon \UU(X) \to \UU(X_{i})$, say $\rho_{1}$, is injective on the image $\phi(\SAff_{n})\subset\UU(X)$, because every nontrivial normal closed subgroup of $\SAff_{n}$ contains the translations.  Let $T_{1}:=\rho_{1}(T') \subset \UU(X_{1})$ be the image of $T'$. Choose a root subgroup $U\subset \UU(X)$ in the image $\phi(\SL_{n}(\CC))$, and denote by $U_{1}:=\rho_{1}(U)$ its image in $\UU(X_{1})$. Then $U$ is a maximal unipotent subgroup of a closed subgroup of $\UU(X)$ isomorphic to $\SLtwo$ which implies that the restriction morphism $\rho_{1}^{*}\colon \OOO(X)^{U} \to \OOO(X_{1})^{U_{1}}$ is surjective. Since the root subgroups of $\UU(X)$ have all different weights this also holds for the root subgroups in $\OOO(X_{1})^{U_{1}}\cdot U_{1}$. Hence the $T_{1}$-action on $\OOO(X_{1})^{U_{1}}$ is multiplicity-free and so $\dim X_{1}\leq n$. Thus $X_{1}\simeq \An$, by Proposition~\ref{Aff.prop}, which implies that $X = X_{1} \simeq \An$, because $X$ is connected.
\end{proof}

\par\medskip
%%%%%%%%%%%%%%%%%%%%%%%%%%%%%
\section{Finite dimensional automorphism groups}
It is well-known that for a smooth affine curve $C$ the automorphism group $\Aut(C)$ is finite except for 
$C\simeq \CC,\Cst$. We will see in the next section that every finite group appears as automorphism group of a smooth affine curve. There also exist examples of smooth affine surfaces with a discrete non-finite automorphism group, see \cite[Proposition~7.2.5]{FuKr2014On-the-geometry-of}. Recall that an ind-group $\GGG = \bigcup_{k}\GGG_{k}$ is called {\it discrete\/} if $\GGG_{k}$ is finite for all $k$, or equivalently, if $\Lie\GGG = \{0\}$. 

\begin{defn} 
An ind-group $\GGG = \bigcup_{k}\GGG_{k}$ is called {\it finite dimensional}, $\dim\GGG < \infty$, if $\dim \GGG_{k}$ is bounded above. In this case we put $\dim\GGG := \max_{k}\dim\GGG_{k}$. 
\end{defn}

\begin{defn}
Let $\GGG = \bigcup_{k}\dim\GGG_{k}$ be an ind-group. Define
$$
\GGG^{\circ}:= \bigcup_{k}\GGG_{k}^{\circ}
$$
where $\GGG_{k}^{\circ}$ denotes the connected  component of $\GGG_{k}$ which contains $e\in\GGG$.
\end{defn}
Recall that an ind-group is called {\it connected} if for every $g\in\GGG$ there is an irreducible curve $D$ and a morphism $D \to \GGG$ whose image contains $e$ and $g$. This implies that for every $k$ there is a $k' \geq k$ and an irreducible component $C \subset \GGG_{k'}$ which contains $\GGG_{k}$ (cf.  \cite[Lemma~8.1.2]{FuKr2014On-the-geometry-of}).

\begin{lem} \lab{connectedcomp.lem}
Let $\GGG = \bigcup_{k}\GGG_{k}$ be an ind-group.
\be
\item $\GGG^{\circ} \subset \GGG$ is a connected closed normal subgroup of countable index. In particular, $\Lie\GGG=\Lie\GGG^{\circ}$.
\item We have  $\dim\GGG < \infty$ if and only if  $\GGG^{\circ}\subset \GGG$ is an algebraic group.
\item We have  $\dim\GGG < \infty$ if and only if $\dim\Lie\GGG < \infty$.
\ee
\end{lem}
\begin{proof}
(1) For $g\in \GGG$ and any $k$ the subset $g \GGG_{k}^{\circ} g^{-1}$ is closed, connected, contains $e$, and is contained in some $\GGG_{k'}$. 
Hence $g \,\GGG_{k}^{\circ} \,g^{-1} \subset \GGG_{k'}^{\circ}\subset\GGG^{\circ}$ which shows that $\GGG^{\circ}$ is normal. If $\GGG^{\circ}$ meets a connected component $C$ of some $\GGG_{k}$, then there is a $k' \geq k$ such that $\GGG_{k'}^{\circ}\cap C \neq\emptyset$, hence $C \subset \GGG_{k'}^{\circ}$. Thus $\GGG^{\circ}\cap\GGG_{k}=\GGG_{k''}^{\circ}\cap\GGG_{k}$ for large $k''\geq k$ and so $\GGG^{\circ}$ is closed. 

To see that $\GGG^{\circ}$ is connected it suffices to show the following. If $C_{0},C_{1}\subset \GGG_{k}^{\circ}$ are irreducible components with $C_{0}\cap C_{1}\neq \emptyset$, then there is a $k'\geq k$ and an irreducible component $C$ of $\GGG_{k'}^{\circ}$ which contains $C_{0}\cup C_{1}$. For this, we choose an $h\in C_{0}\cap C_{1}$ and define the morphism $\mu\colon C_{0}\times C_{1}\to \GGG$ by $(g_{0},g_{1})\mapsto g_{0}h^{-1}g_{1}$. Then the image $Y$ of $\mu$ contains $C_{0}$ and $C_{1}$, and the closure $\bar Y$ is irreducible. Hence, $\bar Y$ is contained in an irreducible component $C$ of $\GGG_{k'}$, and $C \subset \GGG_{k'}^{\circ}$, because $C_{0},C_{1}\subset \GGG_{k}^{\circ}\subset \GGG_{k'}^{\circ}$.

For the last claim we choose elements $g_{1},\ldots,g_{m}\in\GGG_{k}$, one  from each connected component. Then $\GGG_{k}\subset g_{1}\GGG^{\circ}\cup g_{2}\GGG^{\circ}\cup\cdots\cup g_{m}\GGG^{\circ}$ which implies that $\GGG/\GGG^{\circ}$ is countable.

\par\smallskip
(2) Assume that  $\dim \GGG <\infty$.  Then there is a $k_{0}$ such that $\dim\GGG_{k}^{\circ} = \dim\GGG_{k_{0}}^{\circ}$ for all $k\geq k_{0}$. Since $\GGG_{k}^{\circ}$ is contained in an irreducible component $C$ of $\GGG_{k'}^{\circ}$ for some $k'>k$, it follows that $\GGG_{k}^{\circ} = C$ in case $k\geq k_{0}$. Thus $\bigcup_{k}\GGG_{k}^{\circ}= \GGG_{k_{0}}^{\circ}$, and this is a closed irreducible algebraic subset, hence an algebraic group. The reverse implication is clear.

\par\smallskip
(3) If $\dim\GGG < \infty$, then $\GGG^{\circ}$ is an algebraic group, by (2). Since $\Lie\GGG = \Lie\GGG^{\circ}$ by (1) we see that $\Lie\GGG$ is finite dimensional. 
If $\dim\Lie\GGG<\infty$, then $\dim T_{g}\GGG < \infty$, because $T_{g}\GGG_{k} \simeq T_{e}g^{-1}\GGG_{k}\subset \Lie\GGG$. Thus $\dim\GGG_{k}$ is bounded by $\dim\Lie\GGG$.       
\end{proof}

\begin{exa}\lab{torus.exa}
\be
\item 
We have $\Aut(\Cst) \simeq \ZZ/2 \ltimes \Cst$, hence $\Aut(\Cst)^{\circ}\simeq \Cst$. Similarly, $\Aut(\Cst^{n}) \simeq \GL_{n}(\ZZ)\ltimes \Cst^{n}$, and so $\Aut(\Cst^{n})^{\circ}\simeq \Cst^{n}$.
\item
Let $C$ be a smooth curve with trivial automorphism group, and consider the one dimensional variety $Y_{C}=\Aone\cup C$ where the two irreducible components meet in $\{0\}\in\Aone$. Then $\Aut(Y_{C}) \simeq \Cst$. Moreover, the disjoint union $Y_{C_{1}}\cup Y_{C_{2}}\cup\cdots\cup Y_{C_{m}}$ with pairwise non-isomorphic curves $C_{i}$ has automorphism group $\Cst^{m}$. We do not know if there is an irreducible variety whose automorphism group is a given torus.
\ee
\end{exa}

The proof of Theorem~\reff{thm2} follows immediately from the next result.

\begin{prop} Let $X$ be a connected affine variety. If $X$ admits a nontrivial action of the additive group $\Cplus$, then either $X\simeq \Aone$ or $\dim\Aut(X) = \infty$.
\end{prop}
\begin{proof} If $X$ contains a one-dimensional irreducible component $X_{i}$ with a nontrivial action of $\Cplus$, then $X_{i}$ is an orbit under $\Cplus$, hence $X = X_{i} \simeq \Aone$. Otherwise, $\Cplus$ acts non-trivially on an irreducible component $X_{j}$ of dimension $\geq 2$. Denote by $U \subset \Aut(X)$ the image of $\Cplus$. We claim that the unipotent subgroup $\OOO(X)^{U}\cdot U \subset \Aut(X)$ is infinite dimensional. This follows if we show that the image of $\OOO(X)^{U}$ in $\OOO(X_{j})$ is infinite dimensional. For that we first remark that there is a nonzero $U$-invariant $f$ which vanishes on all $X_{j}\cap X_{k}$ for $k\neq j$, because the vanishing ideal is $U$-stable. This implies that $X_{f}\subset X_{j}$, and so
$$
\OOO(X)^{U}_{f} = \OOO(X_{f})^{U} = \OOO(X_{j})^{U}_{f} = (\OOO(X)^{U}|_{X_{j}})_{f}.
$$
Thus the image $\OOO(X)^{U}|_{X_{j}} \subset \OOO(X_{j})$ is infinite dimensional.
\end{proof}

\par\bigskip
%%%%%%%%%%%%%%%%%%%%%%%%%%%%%%
\section{Automorphism groups of curves}
In this last section we prove Theorem~\ref{thm4} which claims that every finite group appear as the automorphism group of a smooth affine curve. 
\begin{lem}\lab{lem3}
Let $C$ be a smooth affine curve which is neither rational nor elliptic, and let $G \subset \Aut(C)$ be a subgroup. Then there is an open subset $C' \subset C$ such that $\Aut(C') = G$.
\end{lem}
\begin{proof}
The curve $C$ is contained as an open set in a smooth projective curve $\bar C$. Since every automorphism $C$ extends to $\bar C$ we have $G \subset \Aut(\bar C)$. By assumption $\bar C$ has genus $>1$ and so $\Aut(\bar C)$ is finite. It follows that there is a point $c \in C$ with trivial stabilizer in $\Aut(\bar C)$. Removing the $G$-orbit $G \cdot c$ we get $\Aut(C \setminus G\cdot c) = G$. In fact, $G \subseteq \Aut(C \setminus G\cdot c)$, and we have equality, because every automorphism of $C \setminus G\cdot c$ extends to $\bar C$.
\end{proof}

The lemma shows that we can prove Theorem~\reff{thm4} by constructing, for every large $n$, a smooth affine curve $C_{n}$ such that the symmetric group $\Sn$ appears as a subgroup of $\Aut(C_{n})$. The following construction was suggested by \name{Jean-Philippe Furter}. We start with the standard action of $\Sn$ on $\Cn$ and consider the morphism
$$
\phi:=(s_{1},s_{2},\ldots,s_{n-1})\colon \Cn \to \CC^{n-1}
$$
where $s_{j}\in\CC[x_{1},\ldots,x_{n}]^{\Sn}$ is the $j$th elementary symmetric function.
The sequence $s_{1},s_{2},\ldots,s_{n}$ is a homogeneous system of parameters, and so the fibers of the morphism $\phi$
are complete intersections of dimension 1 which are stable under the action of $\Sn$. 
Note that for the hyperplane $H:=\VVV(x_{n})$ the induced morphism  $\phi_H:=\phi|_{H}\colon H\to \CC^{n-1}$ is the quotient morphism under the action of $\SSS_{n-1}$.

\begin{prop}\lab{genfiberirred.prop}
The general fiber of $\phi$ is irreducible, i.e. there is a dense open set $U\subset\CC^{n-1}$ such that $\phi^{-1}(u)$ is an irreducible curve for all $u\in U$.
\end{prop}
\begin{proof}
We will give the proof for $n\geq 5$ which is sufficient for our application.
By \cite[Proposition~9.7.8]{Gr1966Elements-de-geomet} we have to show that the generic fiber is geometrically irreducible, or equivalently that $\CC(s_{1},\ldots,s_{n-1})$ is algebraically closed in $\CC(x_{1},\ldots,x_{n})$.
Consider the integral closure $A$ of $\CC[s_{1},\ldots,s_{n-1}]$ in $\Cxn$, and denote by $\eta\colon Y \to \CC^{n-1}$ the corresponding finite morphism. Clearly, we have a factorization $\phi\colon \Cn \overset{\bar\phi}{\to} Y \overset{\eta}{\to} \CC^{n-1}$. Moreover, $A$ is stable under the action of $\Sn$, $A^{\Sn}=\CC[s_{1},\ldots,s_{n-1}]$, and thus $\eta$ is the quotient under the action of $\Sn$ on $Y$. 
Restricting $\bar\phi$ to $H$ we get a similar factorization $\phi_H\colon H\overset{\bar\phi_H}{\to} Y \overset{\eta}{\to} \CC^{n-1}$ with a finite $\SSS_{n-1}$-equivariant morphism $\bar\phi_H$. 
\begin{diagram}
\Cn &&&&\\
\uInto(0,4)&\rdTo_{\bar\phi}\rdTo(4,2)^{\phi}\\
&& Y & \rTo^{\!\!\!\!\!\!\eta} & \CC^{n-1}\\
&\ruTo(2,2)^{\bar\phi_H}&&\ruTo(4,2)_{\phi_H}\\
H\\
\end{diagram}
\par\smallskip\noindent
It follows that  $A^{\SSS_{n-1}}=\CC[s_{1},\ldots,s_{n-1}]=A^{\Sn}$ which implies that the action of $\Sn$ on $Y$ cannot be faithful. Assuming  $n\geq 5$ we deduce that the alternating group $\AAn$ acts trivially on $Y$. 

If $\Sn$ acts trivially, then $Y \simto \CC^{n-1}$ and we are done. Otherwise, $Y = \CC^{n-1}/\AAA_{n-1}$ and we get a decomposition $A = \CC[s_{1},\ldots,s_{n-1}]\oplus \CC[s_{1},\ldots,s_{n-1}]f$ where $f$ is a $\SSS_{n-1}$-semi-invariant and $f^{2}\in\CC[s_{1},\ldots,s_{n-1}]$. 

But $A$ is stable under $\Sn$, and so $f$ is also an $\Sn$-semi-invariant, i.e. $f = d\cdot h$ where $d = \prod_{i<j}(x_{i}-x_{j})$ and $h \in \Cxn^{\Sn}$. Clearly, $f$ vanishes on the hyperplanes $H_{ij}:=\VVV(x_{i}-x_{j})$. On the other hand, we claim that the images $\phi(H_{ij})$ are dense in $\CC^{n-1}$. Hence $f^{2}=0$ on $\CC^{n-1}$, and so $f=0$, a contradiction. 

In order to prove the claim we remark that $\CC[s_{1},\ldots,s_{n-1}] = \CC[p_{1},\ldots,p_{n-1}]$ where $p_{j}:=x_{1}^{j}+\cdots+x_{n}^{j}$. Using these functions to define $\phi\colon\Cn \to \CC^{n-1}$ we see that the Jacobian matrix is given by
$$
\Jac(\phi)=\begin{bmatrix}
1 & 1 & \cdots & 1\\
x_{1} & x_{2} & \cdots & x_{n}\\
x_{1}^{2} & x_{2}^{2} & \cdots & x_{n}^{2}\\
\vdots& \vdots && \vdots \\
x_{1}^{n-2} & x_{2}^{n-2} & \cdots & x_{n}^{n-2}
\end{bmatrix}.
$$
Every $(n-1)\times(n-1)$-minor is a \name{Vandermonde} determinant, and so $\Jac(\phi)$ has rank $n-1$ in all points $a=(a_{1},\ldots,a_{n})$ where at most two coordinates are equal. Hence $\Jac(\phi)$ has maximal rank on a dense open set of every hyperplane $H_{ij}$, and so 
$\phi|_{H_{ij}}$ is dominant.
\end{proof}

\begin{proof}[Proof of Theorem~\ref{thm4}]
We can embed every finite group $G$ into some $\Sn$ where we can assume that $n\geq 5$.
By Proposition~\reff{genfiberirred.prop} the general fiber of $\phi\colon \Cn \to \CC^{n-1}$ is an irreducible curve $C$ with a faithful action of $\Sn$. This action lifts to a faithful action on the normalization $\tilde C$. Since $n\geq 5$, $\tilde C$ is neither rational nor elliptic, hence the claim follows from Lemma~\reff{lem3}.
\end{proof}

\bigskip
%%%%%%%%%%%%%%%%
%\bibliography{HP-Bib.bib}
%\bibliographystyle{amsplain}
%\bibliographystyle{apalike}
%\bibliographystyle{amsalpha}

\par\bigskip\bigskip
\end{document}